\documentclass[a4paper,11pt,leqno]{amsart}
\usepackage[utf8]{inputenc}
\usepackage{hyperref}

\usepackage{amsmath}
\usepackage{amsthm}
\usepackage{amsfonts}
\usepackage{amssymb}

\usepackage{tikz}
\usetikzlibrary{arrows}
\usepackage{enumerate}
\usetikzlibrary{calc}
\usepackage{colonequals}
\usepackage{verbatim}

\usepackage{mathtools}
\usepackage{esint}

\newcommand{\adj}{\operatorname{adj}}
\newcommand{\loc}{\operatorname{loc}}

\newcommand{\R}{\mathbb R}

\DeclareMathOperator{\diam}{diam}
\DeclareMathOperator{\dist}{dist}
\DeclareMathOperator{\card}{card}

\newcommand{\abs}[1]{\lvert#1\rvert}
\newcommand{\norm}[1]{\lVert#1\rVert}
\DeclareMathOperator{\esssup}{ess \, sup}

\newtheorem{theorem}{Theorem}[section]
\newtheorem{conjecture}[theorem]{Conjecture}
\newtheorem*{theorem*}{Theorem}{\bf}{\it}
\newtheorem{proposition}[theorem]{Proposition}
\newtheorem*{proposition*}{Proposition}{\bf}{\it}
\newtheorem{lemma}[theorem]{Lemma}
\newtheorem*{lemma*}{Lemma}{\bf}{\it}

\theoremstyle{definition}

\newtheorem*{definition*}{Definition}
\theoremstyle{remark}
\newtheorem{remark}[theorem]{Remark}

\newtheorem{example}[theorem]{Example}

\numberwithin{equation}{section}

\begin{document}

\title{Remarks on Martio's conjecture}

\author{Ville Tengvall}
\address{Department of Mathematics and Statistics, P.O. Box 68 (Pietari Kalmin katu 5), FI-00014 University of Helsinki, Finland}
\thanks{The research of the author was supported by the Academy of Finland, project number 308759.}

\address{Department of Mathematics and Statistics, University of Jyv\"{a}skyl\"{a}, P.O. Box 35 (MaD), FI-40014 University of Jyv\"{a}skyl\"{a}, Finland}
\email{ville.j.tengvall@jyu.fi}

\date{\today}

\begin{abstract}
We introduce a certain integrability condition for the reciprocal of the Jacobian determinant which guarantees the local homeomorphism property of quasiregular mappings with a small inner dilatation. This condition turns out to be sharp in the planar case. We also show that every branch point of a quasiregular mapping with a small inner dilatation is a Lebesgue point of the differential matrix of the mapping. 
\end{abstract}

\maketitle

\section{Introduction}\label{sec:Intro}

For a given constant $K \ge 1$ a mapping
$$f : \Omega \to \R^n \quad \text{($\Omega \subset \R^n$ domain with $n \ge 2$)}$$
is called \emph{$K$-quasiregular} if it belongs to Sobolev space $W_{loc}^{1,n}(\Omega, \R^n)$ and if the following \emph{distortion inequality}
\begin{align*}
\abs{Df(x)}^n \le K J_f(x)
\end{align*}
holds for almost every $x \in \Omega$. If a mapping is $K$-quasiregular for some $K \ge 1$ then it is called \emph{quasiregular} and if it is, in addition, a homeomorphism we call it \emph{quasiconformal}. Quasiregular and quasiconformal mappings have been intensively studied for several decades and for their basic properties and more detailed background we refer to monographies \cite{Astala-Iwaniec-Martin, HeinonenKilpelainenMartioBook, IwaniecMartin,Reshetnyak67, Rickman-book, Vuorinen, VaisalaBook}. 

To every quasiregular mapping $f : \Omega \to \R^n$ we associate the \emph{inner distortion function}
\begin{align*}
K_I(\cdot, f) : \Omega \to [1, +\infty], \quad K_I(x,f)
=  \left\{ \begin{array}{ll}
\frac{\abs{D^{\#}f(x)}^{n}}{J_f(x)^{n-1}}, & \textrm{if $J_f(x)>0$}\\
1, & \textrm{otherwise,}
\end{array} \right.
\end{align*}
that measures the distortion of the infinitesimal geometry of $n$-dimensional balls under the mapping. We define the corresponding \emph{inner dilatation} in the following way
\begin{align*}
K_I(f) \colonequals \underset{x \in \Omega}{\esssup} \, K_I(x,f) \, .
\end{align*}
Above and in what follows
$$J_f(x) \colonequals \det Df(x) \quad \text{and} \quad D^{\sharp}f(x) \colonequals \adj Df(x)$$
stand for the \emph{Jacobian determinant} and the \emph{adjugate matrix} associated to the differential matrix $Df(x)$. In addition, in the context of matrices $\abs{A}$ denotes the \emph{operator norm} of an $n \times n$ matrix $A$. 

In this article we study the local homeomorphism property of quasiregular mappings by investigating the conditions under which the \emph{branch set} 
$$\mathcal{B}_f = \{ x \in \Omega : \text{$f$ is not a local homeomorphism at $x$} \}$$
of a given quasiregular mapping $f : \Omega \to \R^n$ with a small inner dilatation is empty. Our work is motivated by the well-known \emph{Martio's conjecture} which states that every quasiregular mapping in dimension $n \ge 3$ that has the inner dilatation less than two is a local homeomorphism. The conjecture was originally presented in a paper \cite{MRV-71} by Martio, Rickman, and Väisälä and it was motivated by the preliminary work of Martio \cite{Martio1970}. In this article we state and examine a slightly stronger version of this long-standing open problem:
\begin{conjecture}[Strong Martio's conjecture]\label{StrongMartio}
	A non-constant quasiregular mapping 
	$$f : \Omega \to \R^n \quad (\text{$\Omega \subset \R^n$ domain with $n \ge 3$})$$
	with $K_I(f) < \inf_{x \in \mathcal{B}_f} i(x,f)$ is a local homeomorphism. 
\end{conjecture}
Above and in what follows $i(x,f) \in \mathbb{Z}$ stands for the \emph{local topological index} of a point $x \in \Omega$ under a quasiregular mapping $f : \Omega \to \R^n$. Note that 
$$i(x,f) \ge 2 \quad \text{for every } x \in \mathcal{B}_f \, ,$$
see \cite[Chapter~I]{Rickman-book} for further details. We also point out that the holomorphic function
\begin{align*}
f : \mathbb{C} \to \mathbb{C}, \quad f(z) = z^m \quad \text{($m\ge 2$)}
\end{align*}
shows the conjecture to fail in dimension two. Moreover, if the conjecture is true, then its extremality can be verified by the standard \emph{$m$-to-1 winding mapping} 
\begin{align*}
(r, \theta, z) \stackrel{w}{\mapsto} (r, m\theta, z) \quad \text{($z \in \R^{n-2}$)} \, ,
\end{align*}
written here in cylindrical coordinates. Indeed, this mapping is branching on an $(n-2)$-dimensional hyperplane and for its inner dilatation we have $K_I(w) = m$. 


The main motivation to Martio's conjecture comes from the generalized Liouville's theorem of Gehring \cite{Gehring1962} and Reshetnyak \cite{Reshetnyak1967} which says that non-constant 1-quasiregular mappings in dimension $n \ge 3$ are restrictions of M\"{o}bius transformations, see also \cite{BojarskiIwaniec1982}. In \cite[Theorem~4.6]{MRV-71} Martio, Rickman and V\"{a}is\"{a}l\"{a} (see also \cite{Goldstein1971}) showed that the local homeomorphism property of $1$-quasiregular mappings holds also for mappings with the dilatation close to one by proving that in each dimension $n \ge 3$ there exists $\varepsilon(n)> 0$ such that every quasiregular mapping 
$$f : \Omega \to \R^n \quad (\text{$\Omega \subset \R^n$ domain with $n \ge 3$})$$
with $K_I(f) < 1+\varepsilon(n)$ is a local homeomorphism. However, the proof in \cite{MRV-71} is implicit and does not give any explicit estimate for the number $\varepsilon(n)>0$. The only explicit bound is given by Rajala in \cite{Rajala-MartioResult}, but it is only slightly larger than zero and it depends on the dimension of the space. 

In this article we gain new information on Martio's conjecture by estimating the Hausdorff dimension of the image of the branch set for quasiregular mappings with a small inner dilatation. Our approach applies the modulus of continuity estimate \cite[Theorem~III.4.7]{Rickman-book}, the well-known lower bound \cite[Proposition~III.5.3]{Rickman-book} for the size of the image of the branch set, and the regularity properties of the local inverses of quasiregular mappings defined outside the image of the branch set. The aforementioned regularity properties depend on the level of the integrability of the reciprocal of the Jacobian determinant of the mapping. By denoting
\begin{align*}
\alpha_f \colonequals \biggl( \frac{K_I(f)}{\inf_{x \in \mathcal{B}_f} i(x,f)} \biggr)^{\frac{1}{n-1}} \quad \text{and} \quad q_f \colonequals \frac{2}{n(1-\alpha_f)} - 1 
\end{align*}
we may state our first main result as follows:
\begin{theorem}\label{thm:Main1}
	Let
	$$f : \Omega \to \R^n \quad (\text{$\Omega \subset \R^n$ domain with $n \ge 2$})$$
	be a non-constant quasiregular mapping such that
	$$K_I(f) < \inf_{x \in \mathcal{B}_f} i(x,f) \, .$$
	If the reciprocal of the Jacobian determinant satisfies
	\begin{align*}
	1/J_f \in L_{\loc}^q(\Omega) \quad \text{for some } q > q_f
	\end{align*}
	then $f$ is a local homeomorphism. Moreover, in the planar case the same statement holds even if $q \ge q_f$. 
\end{theorem}

A simple construction shows that the integrability condition of the reciprocal of the Jacobian determinant in Theorem~\ref{thm:Main1} is actually sharp in the planar case, see Example~\ref{ex:Sharpness}. Furthermore, Theorem~\ref{thm:Main1} can be considered as a continuation of the author's earlier joint work \cite{KLT3} with Kauranen and Luisto where the strong Martio's conjecture was verified for \emph{mappings of bounded length distortion} (abbr. \emph{BLD-mappings}). To see this we recall that a mapping of bounded length distortion can be defined as a Lipschitz mapping
$$f : \Omega \to \R^n \quad \text{($\Omega \subset \R^n$ domain with $n \ge 2$)}$$
such that
$$J_f(x) > c \quad \text{for almost every } x \in \Omega \, ,$$
where $c > 0$ is some positive constant. Especially, these mappings form a subclass of quasiregular mappings. Thus, Theorem~\ref{thm:Main1} improves the result in \cite{KLT3} by relaxing the boundedness conditions of the Jacobian determinant. In this article we also obtain several alternative proofs for the local homeomorphism property of BLD-mappings with a small inner dilatation, see Remark~\ref{BLDRemark1} and~\ref{remark:Dimension}. For further information about the local injectivity properties of BLD-mappings we refer to \cite{GutlyanskiMartioRyazanovVuorinen2000,HeinonenKilpelainen2000, MartioVaisala}. All these results should be contrasted with the well-known inverse function theorem which states that every continuously differentiable mapping is a local diffeomorphism outside its zero set of the Jacobian determinant, see e.g. \cite{KrantzParks2013} for some further details on this well-known result.

There are also several other important partial results in the direction of Martio's conjecture. First of all, the conjecture is known to be true if the branch set contains any non-constant rectifiable arc, see e.g. \cite[p.~76]{Rickman-book}. Also every sufficiently smooth quasiregular mapping is a local homeomorphisms even if its inner dilatation is not considered small, see \cite{BonkHeinonen-Smooth, KaufmanTysonWu2005}. Local injectivity of quasiregular mappings also follows if the \emph{dilatation tensor}
\begin{displaymath}
G_f(x) = \left\{ \begin{array}{ll}
\frac{Df(x)^T Df(x)}{J_f(x)^{2/n}}, & \textrm{if $J_f(x)>0$}\\
\text{I}\, , & \textrm{otherwise},
\end{array} \right.
\end{displaymath}
is approximately continuous or is close to some continuous matrix-valued function in BMO or VMO, see \cite[Theorem~3.22]{GutlyanskiMartioRyazanovVuorinen1998} and \cite[Theorem~4.1 and~5.1]{MartioRyazanovVuorinen1999}. Most of the arguments that apply the properties of the dilatation tensor to study Martio's conjecture use a similar kind of implicit compactness argument as in \cite{MRV-71} combined with the fact that a quasiregular mapping is locally invertible at each point $x_0 \in \Omega$ where the condition
\begin{align*}
\lim_{r \to 0} \fint_{B(x_0,r)} \abs{G_f(z)-\fint_{B(x_0,r)} G_f(x) \, dx} \, dz = 0 
\end{align*}
holds, provided that $n \ge 3$. Especially, it follows that if $G_f \in W^{1,p}$ for some $p \in [1,n]$ then
\begin{align}\label{eq:SizeOfBranch}
\mathcal{H}^{n-p}(\mathcal{B}_f) = 0.
\end{align}
For further discussion in this direction we refer to \cite{HeinonenKilpelainen2000}. If one wishes to study Martio's conjecture in terms of the dilatation tensor then the following theorem might turn out to be useful as we see in Remark~\ref{BLDRemark1}.(3):

\begin{theorem}\label{thm:Main2}
Let
$$f : \Omega \to \R^n \quad (\text{$\Omega \subset \R^n$ domain with $n \ge 2$})$$
be a non-constant $K$-quasiregular mapping such that
$$K_I(f) < i(x_0,f) \quad \text{for a given } x_0 \in \mathcal{B}_f.$$
Then there exist 
$$p = p(n,K) > n \quad \text{and} \quad r_0>0$$ 
such that for every exponent $0 < s < p(n,K)$ and every radius $0 < r < r_0$ we have
$$\fint_{B(x_0,r)} \abs{Df(x)-Df(x_0)}^s \, dx \le C_{x_0} r^{s(\mu-1)},$$
where $\mu = (i(x_0,f)/K_I(f))^{1/(n-1)}$ and $C_{x_0} \colonequals C_{x_0}(n,K,s)$.
\end{theorem}
In addition to the above-mentioned results, some additional information and related results to Martio's conjecture can be found from \cite{HeinonenICM, Iwaniec1987, IwaniecMartin1993, MartioRickman, Rickman85} and \cite{QuasiconformalSpaceMappings}.

\section*{\textbf{Acknowledgments}} The author wishes to thank Jani Onninen for introducing him the technique applied here to prove Proposition~\ref{lemma:QRMartioLebesgue}. In addition, the author wishes to express his gratitude to Kai Rajala for pointing out that Theorem~\ref{thm:Main1} is valid also in the planar case as well as for commenting earlier versions of this article. Finally, he would like to thank Xiao Zhong for his support and encouragement throughout the project as well as Aapo Kauranen and Rami Luisto for several fruitful discussions during the project.

\section{Decay of the differential matrix at the branch points and the proof of Theorem~\ref{thm:Main2}}\label{sec:LebesguePoints}

Recall that if
$$f : \Omega \to \R^n \quad \text{($\Omega \subset \R^n$ domain with $n \ge 2$)}$$
is a non-constant quasiregular mapping then by \cite[Theorem~III.4.7]{Rickman-book} for each $x \in \Omega$ there exist positive numbers $r(x)$ and $C_I(x)$ such that
\begin{align}\label{ModulusOfContinuity}
\abs{f(x)-f(y)} \le C_I(x) \abs{x-y}^{1/\alpha_f(x)} \quad \text{for every } y \in B(x,r(x)) \, ,
\end{align}
where $\alpha_f(x) \colonequals \bigl( K_I(f)/i(x,f) \bigr)^{1/(n-1)}$. When 
$$K_I(f) < i(x,f) \quad \text{for all $x \in \mathcal{B}_f$} \, ,$$ 
the upper estimate in \eqref{ModulusOfContinuity} implies that
\begin{align}\label{BranchSetInclusion}
\mathcal{B}_f \subset \mathcal{D}_f \cap \mathcal{C}_f ,
\end{align}
where
\begin{align*}
\mathcal{D}_f \colonequals \{ x \in \Omega : f \text{ is differentiable at } x \}
\end{align*}
is the \emph{set of differentiability of $f$} and
\begin{align*}
\mathcal{C}_f \colonequals \{ x \in \Omega : \abs{Df(x)} = 0 \}
\end{align*}
is the \emph{critical set of $f$}. Usually the set $\mathcal{D}_f \cap \mathcal{C}_f$ is much larger than the branch set $\mathcal{B}_f$, and this is the case even for quasiregular mappings with a small dilatation. Indeed, by \cite[Final remarks 23]{GehringVaisala1973} for every $\varepsilon > 0$ there exists a $(1+\varepsilon)$-quasiconformal mapping $f : \Omega \to \R^n$ such that the Hausdorff dimension of the intersection $\mathcal{D}_f \cap \mathcal{C}_f$ satisfies 
$$\dim_H(\mathcal{D}_f \cap \mathcal{C}_f) = n.$$
On the other hand, by \cite[Theorem~1.3]{BonkHeinonen-Smooth} there exists $\eta(n,K)>0$ such that
$$\dim_H \mathcal{B}_f \le n-\eta(n,K)$$
for every $K$-quasiregular mapping $f : \Omega \to \R^n$, see also \cite{OnninenRajala}. In other words, under the assuptions of the strong Martio's conjecture the inclusion in \eqref{BranchSetInclusion} could in principle be strict. Therefore, it is not enough to analyze the critical set in order to provide information about the geometric properties of the branch set for quasiregular mappings with a small dilatation. 

Despite these difficulties it turns out that under the assumptions of the strong Martio's conjecture we can still estimate the decay of $\abs{Df}$ at the branch points. For this purpose, let us recall that by \cite[Theorem~14.4.1]{IwaniecMartin} every $K$-quasiregular mapping
$$f : \Omega \to \R^n \quad \text{($\Omega \subset \R^n$ domain with $n \ge 2$)}$$
belongs to Sobolev space $W_{\loc}^{1,s}(\Omega, \R^n)$ for all $s$ with
\begin{align*}
q(n,K) \colonequals \frac{n \lambda K}{\lambda K+1} < s < \frac{n \lambda K}{\lambda K-1} =: p(n,K),
\end{align*}
for some $\lambda = \lambda(n) \ge 1$. The constant $\lambda$ is commonly conjectured to equal one, see e.g. \cite[p. 164]{IwaniecMartin}. Furthermore, for each test function $\phi \in C_0^{\infty}(\Omega)$ the  Caccioppoli-type inequality 
\begin{align*}
\norm{\phi Df}_{L^s} \le C(n,K, s) \norm{f \otimes \nabla \phi}_{L^s}
\end{align*}
holds whenever $q(n,K)< s < p(n,K)$. This way we obtain the following result.

\begin{proposition}\label{lemma:QRMartioLebesgue}
	Let 
	$$f : \Omega \to \R^n \quad (\text{$\Omega \subset \R^n$ domain with $n \ge 2$})$$ 
	be a non-constant $K$-quasiregular mapping such that 
	$$K_I(f) < i(x_0,f) \quad \text{at a point } x_0 \in \mathcal{B}_f$$
	and suppose that $q(n,K) < s < p(n,K)$.
	Then there exist constants 
	$$C_{x_0} \colonequals C_{x_0}(K,n,s)>0 \quad \text{and} \quad r_0 \colonequals r_0(x_0) > 0$$ 
	such that
	\begin{align*}
	\fint_{B(x_0,r)} \abs{Df(x_0)-Df(x)}^s  \, dx \le C_{x_0}(K,n,s) r^{s(\mu-1)} \quad \text{whenever } 0 < r < r_0 \, ,
	\end{align*}
	where $\mu \colonequals (i(x_0,f)/K_I(f))^{1/(n-1)}$.
\end{proposition}

\begin{proof}
	By \eqref{ModulusOfContinuity} the mapping $f$ is differentiable at the point $x_0 \in \mathcal{B}_f$ and 
	$$Df(x_0) = 0.$$ 
	Without loss of generality, we may assume that $x_0=f(x_0)=0$. For each $r>0$ define $\varphi_r : \R^n \to \R$,
	\begin{displaymath}
	\varphi_r(x) = \left\{ \begin{array}{ll}
	1, & \textrm{if $0 \le \abs{x} \le r$}\\
	1-\frac{\abs{x}-r}{r}, & \textrm{if $r < \abs{x} < 2r$}\\
	0, & \textrm{if $\abs{x} \ge 2r$}
	\end{array} \right.
	\end{displaymath}
	to be the standard cut-off function. If we fix an exponent $s$ such that
	$$q(n,K) < s < p(n,K),$$
	then by \cite[Theorem~14.4.1]{IwaniecMartin} we get
	\begin{align}\label{eq:Lebesgue1}
	\int_{B_{2r}} \abs{Df(x_0)-Df(x)}^s \, dx &\le C(K,n,s ) \int_{B_{2r}} \abs{f(x)}^s \abs{\nabla \varphi(x)}^s \, dx\\
	&= C(K,n,s)r^{-s}\int_{B_{2r}} \abs{f(x)}^s  \, dx. \nonumber
	\end{align}
	By \cite[Theorem~III.4.7]{Rickman-book} we find $C_{x_0}>0$ and $r_0>0$ such that
	\begin{align}\label{eq:Lebesgue2}
	\abs{f(x)} \le C_{x_0} \abs{x}^{\mu}
	\end{align}
	for all $x \in B(0,r_0)$. Therefore, by combining \eqref{eq:Lebesgue1} and \eqref{eq:Lebesgue2} we have
	\begin{align*}
	\fint_{B_{r}} \abs{Df(x_0)-Df(x)}^s \, dx &\le C_{x_0}(K,n,s)  r^{s(\mu-1)} \quad \text{for all } x \in B(x_0, r_0) \, ,
	\end{align*}
	and the claim follows.
\end{proof}

Now Theorem~\ref{thm:Main2} is a direct consequence of Proposition~\ref{lemma:QRMartioLebesgue} and H\"{o}lder's inequality.

\begin{proof}[Proof of Theorem~\ref{thm:Main2}]
Fix $s \in (0,p(n,K))$. We may assume that 
$$0 < s \le q(n,K)$$ 
as otherwise the claim follows from Proposition~\ref{lemma:QRMartioLebesgue}. By applying Hölder's inequality and Proposition~\ref{lemma:QRMartioLebesgue} we get
\begin{align*}
\fint_{B(x_0,r)} \abs{Df(x_0)-Df(x)}^s  \, dx &\le  \Biggl(\fint_{B(x_0,r)} \abs{Df(x_0)-Df(x)}^{n}  \, dx \biggr)^{\frac{s}{n}} \\
&\le  C_{x_0}(K,n,s) r^{s(\mu-1)},
\end{align*}
and the claim follows.
\end{proof}

Note that as a consequence of Theorem~\ref{thm:Main2} every branch point of a given non-constant quasiregular mapping with the assumptions of the strong Martio's conjecture is a Lebesgue point of the differential matrix of the mapping. In the BLD-setting this is already enough to prove the conjecture as we see in the following remark.

\begin{remark}\label{BLDRemark1} Suppose that
	$$f : \Omega \to \R^n \quad \text{($\Omega \subset \R^n$ domain with $n \ge 2$)}$$
	is a mapping with a bounded length distortion such that 
	$$K_I(f) < \inf_{x \in \mathcal{B}_f} i(x,f) \, .$$
	Then we observe the following.
	\begin{itemize}
		\item[(1)] One can use the upper estimate in \eqref{ModulusOfContinuity} and radiality properties of BLD-mappings to show that $f$ is a local homemorphism, see \cite{KLT3}.
		\item[(2)] Proposition~\ref{lemma:QRMartioLebesgue} implies the local homeomorphism property of $f$. Indeed, for every BLD-mapping we may find a constant $c>0$ such that
		$$J_f(x) > c \quad \text{for almost every } x \in \Omega.$$
		Therefore, Proposition~\ref{lemma:QRMartioLebesgue} can hold for $f$ only if $\mathcal{B}_f = \emptyset$.
		\item[(3)] If $x_0 \in \mathcal{B}_f$ then we may use Theorem~\ref{thm:Main2} to show 
		\begin{align*}
		\lim_{r \to 0} \fint_{B(x_0,r)} \abs{G_f(x)} \, dx = 0, 
		\end{align*}
		which contradicts the fact that $\abs{G_f(x)} \ge 1$ almost everywhere. This is also enough to show that $\mathcal{B}_f = \emptyset$.
	\end{itemize}
Especially, each of the observations above imply the strong Martio's conjecture for BLD-mappings.
\end{remark}

\section{Proof of Theorem~\ref{thm:Main1} and an example for its sharpness in the planar case}

In order to prove Theorem~\ref{thm:Main1} let us first introduce the notation and the terminology used in the proof. We start by recalling that a domain $U \subset \subset \Omega$ is called a \emph{normal domain} of a quasiregular mapping
$$f : \Omega \to \R^n \quad \text{($\Omega \subset \R^n$ domain with $n \ge 2$)}$$
if
$$f(\partial U) = \partial f(U) \, .$$
If a normal domain $U$ satisfies
$$U \cap f^{-1}(f(x)) = \{ x\},$$
then it is called a \emph{normal neighborhood} of $x$. In addition, for a given point $x \in \Omega$ we denote by $U(x,f,r)$ the $x$-component of the preimage $f^{-1}(B(f(x),r))$. The following standard lemma from \cite[Lemma~I.4.9]{Rickman-book} is needed for the proof.
\begin{lemma}\label{Lemma:NormalDomain}
	Let
	$$f : \Omega \to \R^n \quad (\text{$\Omega \subset \R^n$ domain with $n \ge 2$})$$
	be a non-constant quasiregular mapping. Then for every $x \in \Omega$ there exists a radius $r_x > 0$ for which $U(x,f,r)$ is a normal neighborhood of $x$ such that
	$$f(U(x,f,r)) = B(f(x),r) \quad \text{for every } 0 < r \le r_x.$$
	Moreover, we have
	$$\diam U(x,f,r) \to 0 \quad \text{as } r \to 0.$$
	One may also replace the balls in the statement by cubes.
\end{lemma}

\begin{proof}[Proof of Theorem~\ref{thm:Main1}] We prove the planar and the higher dimensional cases separately. Note that, beacause we do not have a good global control on the constants $C_I(x) > 0$ in the modulus of continuity estimate \eqref{ModulusOfContinuity}, our proof gives a slightly weaker outcome in higher dimensions.
	
\subsubsection*{A) Proof of the planar case:} Let us assume that the inner dilatation of a non-constant quasiregular mapping
$$f : \Omega \to \R^2 \quad \text{($\Omega \subset \R^2$ domain)}$$
satisfies the assumption
$$K_I(f) < \inf_{x \in \mathcal{B}_f} i(x,f) \, ,$$
and that for the reciprocal of the Jacobian determinant we have
$$1/J_f \in L_{\loc}^q(\Omega) \quad \text{for some } q \ge q_f \, .$$
For the contradiction suppose that $\mathcal{B}_f \neq \emptyset$. Without loss of generality, we may assume that
$$0 \in \mathcal{B}_f \quad \text{and} \quad f(0)=0, \quad \text{and denote} \quad m \colonequals i(0,f).$$
By~\eqref{ModulusOfContinuity} and Lemma~\ref{Lemma:NormalDomain} there exist a radius $r_0 > 0$ and a constant $C_0 > 0$ such that  $U(0,f,r)$ is a normal neighborhood of the origin for every radius $0 < r \le r_0$ and	
\begin{align}\label{eq:ContinuityPlanarProof}
\abs{f(x)} \le C_0 \abs{x}^{1/\alpha_f} \quad \text{whenever } \abs{x} < r_0.
\end{align}
Moreover, by Sto\"{\i}low factorization theorem the branch set of a planar quasiregular mapping is always a discrete set and therefore we may assume that 
$$\mathcal{B}_f \cap U(0,f,r) = \{ 0 \}$$ 
for all $0 < r \le r_0$. Especially, if we denote $\tilde{r}_0 \colonequals r_0^{1/\alpha_f}$ then it follows from \eqref{eq:ContinuityPlanarProof} that
\begin{align}\label{eq:SomethingNewForMe}
B(0,r^{\alpha_f}/C_0^{\alpha_f}) \subset U(0,f,r) \quad \text{whenever } 0 < r < \tilde{r}_0.
\end{align} 

Next, fix a radius 
$$0 < r < \min \{ \tilde{r}_0, 1 \}$$
and denote $B_r \colonequals B(0,r)$. By applying Lemma~\ref{Lemma:NormalDomain} we cover $B_r \setminus \{ 0\}$ by a countable collection of closed squares $\{\overline{Q}_i \}_{i=1}^{\infty}$ with pairwise disjoint interiors and such that
$$f^{-1}(Q_i) = U_{i,1} \cup \cdots \cup U_{i,m},$$
where the sets $U_{i,j} \subset \subset U(0,f,r) \setminus \{ 0\}$ are pairwise disjoint normal domains of $f$, see the proof of \cite[Lemma~II.7.1]{Rickman-book} for further details of this step. Then the restricted mappings
$$f_{i,j} \colonequals f|_{U_{i,j}} : U_{i,j} \to Q_i \quad \text{($i=1,2, \ldots$ and $j=1,\ldots, m$)}$$
are quasiconformal homeomorphisms. Especially, by~\eqref{eq:SomethingNewForMe} we may estimate the measure of the set $U(0,f,r)$ from below as follows
\begin{align}\label{eq:PlanarEquationOfDoom}
C \sum_{i,j} \int_{B_r} J_{f_{i,j}^{-1}}(y) \chi_{Q_i}(y) \, dy = C \mathcal{H}^2(U(0,f,r)) \ge r^{\alpha_f}, 
\end{align}
where $\chi_{Q_i}$ stands for the standard characteristic function of the square $Q_i$ and $C>0$ is some absolute constant. 

Now it follows from the assumption $q \ge q_f$, Hölder's inequality, estimate \eqref{eq:PlanarEquationOfDoom}, chain rule, and the change of variable formula that
\begin{align}\label{eq:MainCalculationPlanar}
1 \le r^{2(\alpha_f(q+1)-q)} &= C r^{-2q} \abs{B(0,r^{\alpha_f})}^{q+1} \nonumber \\
&\le C r^{-2q} \biggl( \sum_{i=1}^{\infty} \sum_{j=1}^{m} \int_{B_r} J_{f_{i,j}^{-1}}(y) \chi_{Q_i}(y) \, dy \biggr)^{q+1} \nonumber \\
& \le C  \sum_{i=1}^{\infty} \sum_{j=1}^{m} \int_{B_r} J_{f_{i,j}^{-1}}(y)^{q+1} \chi_{Q_i}(y) \, dy \\
&= C \sum_{i=1}^{\infty} \sum_{j=1}^{m} \int_{B_r} \frac{1}{J_{f}(f_{i,j}^{-1}(y))^q} J_{f_{i,j}^{-1}}(y) \chi_{Q_i}(y) \, dy \nonumber \\
&= C \sum_{i=1}^{\infty} \sum_{j=1}^{m} \int_{U_{i,j}} \frac{1}{J_f(x)^q}  \, dx \nonumber \\
&= C \int_{U(0,f,r)} \frac{1}{J_f(x)^q}  \, dx, \nonumber
\end{align}
where the constant $C>0$ varies line by line but is independent on $r>0$. By applying Lemma~\ref{Lemma:NormalDomain} and the absolute continuity of Lebesgue integral we see that 
$$\int_{U(0,f,r)} \frac{1}{J_f(x)^q}  \, dx \stackrel{r \to 0}{\to} 0 \, ,$$
which leads to a contradiction with the estimate \eqref{eq:MainCalculationPlanar}. This proves the theorem in the plane.
	
\subsubsection*{B) Proof of the higher dimensional case:} Let us assume that the inner dilatation of a non-constant quasiregular mapping
$$f : \Omega \to \R^n \quad \text{($\Omega \subset \R^n$ domain with $n \ge 3$)}$$
satisfies the assumption
$$K_I(f) < \inf_{x \in \mathcal{B}_f} i(x,f) \, ,$$
and that for the reciprocal of the Jacobian determinant we have
$$1/J_f \in L_{\loc}^q(\Omega) \quad \text{for some } q > q_f \, .$$
For the contradiction assume that $\mathcal{B}_f \neq \emptyset$. First we obtain that we may always assume
$$q_f = \frac{2}{n(1-\alpha_f)}-1 > 0 \, .$$
Indeed, otherwise it would follow that
$$\inf_{x \in \mathcal{B}_f} i(x,f) \ge \biggl( \frac{n}{n-2} \biggr)^{n-1} K_I(f) \, ,$$
which is not possible if $\mathcal{B}_f \neq \emptyset$, see \cite[Corollary~III.5.8]{Rickman-book}. Thus, from now on we assume
$$0 < q_f < q \, .$$ 
From now on let us denote 
\begin{align*}
\alpha \colonequals n(\alpha_f(q+1)-q) < n-2
\end{align*}
and fix $\varepsilon > 0$ such that
\begin{align}\label{DefOfEpsilon}
\alpha + \varepsilon n(q+1) = n-2 \, .
\end{align}

Next, fix a point $x_0$ from the \emph{inf-index branch set} 
$$\mathcal{B}_f^I \colonequals \{ x \in \Omega : i(x,f) = \inf_{z \in \mathcal{B}_f} i(z,f) \} $$ 
and fix a normal neighborhood $U$ of $x_0$ such that
$$i(x_0,f) = N(f,U) \colonequals \sup_{y \in \R^n} \card f^{-1}(y) \cap U ,$$
see \cite[Proposition I.4.10]{Rickman-book} for the existence of such a normal neighborhood. Then it follows that 
\begin{align*}
i(x,f) = i(x_0, f) \quad \text{for every } x \in \mathcal{B}_f \cap U,
\end{align*}
see again \cite[Proposition~I.4.10]{Rickman-book}. By applying \cite[Proposition~III.5.3]{Rickman-book} we may find a compact set $F \subset \mathcal{B}_f \cap U$ such that
\begin{align}\label{eq:MeasureOfTheHelpSet}
\mathcal{H}^{n-2}(f(F)) > 0 \, .
\end{align}
Fix $0 < \delta < 1$. Then by applying Lemma~\ref{Lemma:NormalDomain} and the modulus of continuity estimate \eqref{ModulusOfContinuity} we cover the set $f(F)$ by a collection of balls $\{B(y_{\eta},r_{\eta}) \}_{\eta \in I}$ satisfying the following conditions:
\begin{itemize}
	\item[(i)] For each $\eta \in I$ we have $y_{\eta} \in f(F)$.
	\item[(ii)] For each $\eta \in I$ we have
	$$B(x_{\eta},r_{\eta}^{\alpha_f + \varepsilon}) \subset U(x_{\eta},f,r_{\eta})$$ 
	where $\varepsilon > 0$ is defined in \eqref{DefOfEpsilon}. Note that obtaining this condition would not be possible if $\varepsilon = 0$, i.e., if $q = q_f$. 
	\item[(iii)] For each $\eta \in I$ we have
	$$0 < r_{\eta} < \delta \quad \text{and} \quad \diam(U(x_{\eta},f,r_{\eta})) < \delta.$$
\end{itemize}
By Vitali covering lemma we then find a countable subcover $B_1, B_2,\ldots$ of pairwise disjoint balls such that
$$f(F) \subset \bigcup_{k} 5B_k \, .$$
For this subcover we write each set $B_k \setminus f(\mathcal{B}_f \cap U)$ as a countable union 
$$B_k \setminus f(\mathcal{B}_f \cap U) = \bigcup_{i=1}^{\infty} \overline{Q}_{i_k}$$
of closed cubes $\overline{Q}_{i_k}$ with non-empty, pairwise disjoint interiors such that
$$f^{-1}(Q_{i_k}) = U_{i_k,1} \cup \cdots \cup U_{i_k,m} \, ,$$
where the sets $U_{i_k,j_k}\subset U_k \setminus \mathcal{B}_f$ are pairwise disjoint normal domains. The existence of such a collection of cubes $Q_{i_k}$ and normal domains $U_{i_k,j_k}$ can be verified by applying Lemma~\ref{Lemma:NormalDomain} similarly as we did in the planar case. 

Now we observe that each of the restrictions
$$f_{i_k,j_k} \colonequals f|_{U_{i_k, j_k}} : U_{i_k,j_k} \to Q_{i_k} \quad \text{($i_k =1,2, \ldots$ and $j_k=1, \ldots, m$)}$$
gives a quasiconformal inverse 
$$f_{i_k,j_k}^{-1} : Q_{i_k} \to U_{i_k,j_k} \, .$$
In addition, as a quasiregular mapping the mapping $f$ satisfies \emph{Lusin's contition $(N^{-1})$} which means that for every set $E \subset \Omega$ we have
$$\mathcal{H}^n(E)= 0 \quad \text{whenever } \mathcal{H}^n(f(E)) = 0 \, ,$$
and therefore we may assume that 
$$\mathcal{H}^n(\partial U_{i_k,j_k})= 0 \quad \text{for every index pair } i_k,j_k \, .$$
Moreover, we also have  $\mathcal{H}^n(\mathcal{B}_f) = 0$, see \cite[Theorem~II.7.4]{Rickman-book}. Thus, by applying area formula we obtain that
\begin{align}\label{eq:JacobianAndNormalDomain}
C\sum_{i_k} \sum_{j_k} \int_{Q_{i_k}} J_{f_{i_k,j_k}^{-1}}(y) \, dy =   C\mathcal{H}^n(U_k) \ge   r_k^{n(\alpha_f+\varepsilon)},
\end{align}
where $U_k \colonequals U(x_k, f, r_k)$ and $C>0$ is some absolute constant. By applying estimate \eqref{eq:JacobianAndNormalDomain}, Hölder's inequality, chain rule, and the change of variable formula we get
\begin{align}\label{eq:MainCalculation}
\sum_{k} r_k^{n(\alpha_f(q+1)-q) + \varepsilon n(q+1)} &\le C \sum_{k} r_k^{-nq} \biggl( \sum_{i_k} \sum_{j_k} \int_{Q_{i_k}} J_{f_{i_k,j_k}^{-1}}(y)  \, dy \biggr)^{q+1}  \nonumber
\\ \nonumber
& \le C  \sum_{k} \sum_{i_k} \sum_{j_k} \int_{Q_{i_k}} J_{f_{i_k,j_k}^{-1}}(y)^{q+1} \, dy  \\
&= C \sum_{k} \sum_{i_k} \sum_{j_k} \int_{Q_{i_k}} \frac{1}{J_{f}(f_{i_k,j_k}^{-1}(y))^q} J_{f_{i_k,j_k}^{-1}}(y) \, dy  \\  \nonumber
&= C \sum_{k} \sum_{i_k} \sum_{j_k} \int_{U_{i_k,j_k}} \frac{1}{J_f(x)^q}  \, dx \\  \nonumber
&= C \sum_{k}  \int_{U(y_k,f,r_k)} \frac{1}{J_f(x)^q}  \, dx\\  \nonumber
&\le C  \int_{\{x : \dist(x,\mathcal{B}_f \cap U) < 2\delta \}} \frac{1}{J_f(x)^q}  \, dx
\end{align}
where the constant $C>0$ varies line by line but is independent on the parameter $\delta>0$. 

In order to finish the proof we recall that $\varepsilon > 0$ was chosen such a way that
$$n(\alpha_f(q+1)-q) + \varepsilon n(q+1) = \alpha + \varepsilon n(q+1) = n-2 \, .$$
Therefore, estimate \eqref{eq:MainCalculation} gives for the $(n-2)$-dimensional Hausdorff $\delta$-content of the set $f(F)$ the following
\begin{align*}
\mathcal{H}_{\delta}^{n-2}(f(F)) \le  \int_{\{x : \dist(x,\mathcal{B}_f) < 2\delta \}} \frac{1}{J_f(x)^q}  \, dx \stackrel{\delta \to 0}{\to} 0 \, ,
\end{align*}
where the convergence to the zero on the right-hand side follows from the absolute continuity of the Lebesgue integral, from the local compactness of the branch set $\mathcal{B}_f$, and the fact that $\mathcal{H}^n(\mathcal{B}_f) = 0$. In this way we see that $\mathcal{H}^{n-2}(f(F)) = 0$ which contradicts the assuption \eqref{eq:MeasureOfTheHelpSet}. This concludes the proof.
\end{proof}

\begin{example}\label{ex:Sharpness}
Consider the function
$$f : \mathbb{D} \to \mathbb{C}, \quad f(z) = z^m \abs{z}^{-\frac{m(K-1)}{K}} \quad \quad \text{($\mathbb{D} \subset \mathbb{C}$ unit disk)} \, ,$$
where $m \ge 2$ is a given integer and $1 \le K < m$ is some real number. In polar coordinates the function takes the form
\begin{align*}
(r, \theta) \mapsto (r^{m/K}, m\theta) \, .
\end{align*}	
Therefore, it is an $m$-to-1 quasiregular mapping such that $0 \in \mathcal{B}_f$ and with the distortion
\begin{align*}
\frac{\abs{D^{\#}f}^2}{J_f} = \frac{m r^{\frac{m-K}{K}}}{(m/K)r^{\frac{m-K}{K}}} = K \quad \text{a.e.}
\end{align*}
Especially, it follows that our function is $K$-quasiregular. Furthermore, a direct computation gives us
\begin{align}\label{eq:JacobianOfTheCounterexample}
\int_{\mathbb{D}} \frac{1}{J_f(z)^q} \, dz &= \frac{2\pi K }{m^2} \int_0^1 \frac{t}{t^{\frac{2q(m-K)}{K}}} \, dt \, .
\end{align}
As the integral~\eqref{eq:JacobianOfTheCounterexample} is finite if and only if $q < \frac{K}{m-K}$ it follows that
$$1/J_f \in L^q(\mathbb{D}) \quad \text{for all} \quad q < q_f \, .$$
In particular, this demonstrates the sharpness of Theorem~\ref{thm:Main1} in the  plane.
\end{example}

\section{Final remarks}

We end this article with a few remarks which we consider to be useful for the further studies on Martio's conjecture. In what follows, let 
$$f : \Omega \to \R^n \quad \text{($\Omega \subset \R^n$ domain with $n \ge 2$)}$$
be a non-constant quasiregular mapping with a non-empty branch set. We then point out that with some minor modifications in the proof of \cite[Theorem~III.5.5]{Rickman-book} it is possible to show the following upper bound for the Hausdorff dimension of the image of the branch set
\begin{align}\label{eq:DimensionEstimate}
\dim_H f(B_f) \le \alpha_f \dim_H B_f \, .
\end{align}
This shows that in dimension $n \ge 3$ a non-constant quasiregular mapping with a small inner dilatation always reduces the Hausdorff dimension of its branch set. This leads to the following two remarks.

\begin{remark}\label{remark:Dimension}
By \cite[Corollary~4.23]{MartioVaisala} BLD-mappings preserve Hausdorff dimensions of sets. Thus the dimension estimate \eqref{eq:DimensionEstimate} provides us yet another proof
for the strong Martio's conjecture in the BLD-setting.
\end{remark}

\begin{remark}
It follows from \cite[Proposition~III.5.3]{Rickman-book} combined with the estimates \eqref{eq:SizeOfBranch} and \eqref{eq:DimensionEstimate} that every non-constant quasiregular mapping
$$f : \Omega \to \R^n \quad \text{($\Omega \subset \R^n$ domain with $n \ge 3$)}$$
with the inner dilatation 
$$K_I(f) < \inf_{x \in \mathcal{B}_f} i(x,f)$$
and with the dilatation tensor
$$G_f \in W^{1,p} \quad \quad \text{for some } p > \frac{2-n(1-\alpha_f)}{\alpha_f}$$
is a local homeomorphism. This observation is motivated by the result \cite[Theorem~3.3]{HeinonenKilpelainen2000} of Heinonen and Kilpeläinen.
\end{remark}

Finally, we point out that the results of this article connect Martio's conjecture to the problem about the optimal integrability of the reciprocal of the Jacobian determinant for quasiregular mappings:

\begin{remark}\label{Remark:IntegrabilityOfJacobian}
	Suppose that
	$$f : \Omega \to \R^n \quad \text{($\Omega \subset \R^n$ domain with $n \ge 2$)}$$
	is a non-constant quasiregular mapping. Then by following the fundamental works of Gehring \cite{Gehring1973}, Elecrat and Meyers \cite{ElecratMeyers1975}, and Martio \cite{Martio1974} and applying theory of Muckenhoupt weights, see e.g. \cite[Chapter~V]{SteinBook}, it is possible to show that
	$$1/J_f \in L_{\loc}^q(\Omega) \quad \text{for some $q>0$}.$$
	However, the optimal level of integrability for $1/J_f$ is only known in the planar case where it is a consequence of the deep work \cite{Astala1994} by Astala.  We refer to \cite[p.~254]{HenclKoskelaZhong2007} and \cite[p.~964]{KoskelaOnninenRajala2012} for some further discussion in this direction.
\end{remark}


\newcommand{\etalchar}[1]{$^{#1}$}
\def\cprime{$'$}

\end{document}